%% file: Preprint.tex
\documentclass{article}

\usepackage{arxiv}

\usepackage{amsmath, cite, amssymb, amsthm}

\usepackage[utf8]{inputenc} 
\usepackage[T1]{fontenc}    

\usepackage{hyperref}       
\usepackage{url}            
\usepackage{amsfonts}       
\usepackage{graphicx}
\usepackage{xcolor,comment}

\usepackage{indentfirst}
\title{Polynomial Convergence of an Observer for an Infinite-Dimensional Oscillating System}

\author{
 A.~Zuyev \\
  Institute of Applied Mathematics and Mechanics, Sloviansk, Ukraine, \\
  Max Planck Institute for Dynamics of Complex Technical Systems, Magdeburg, Germany \\
  \texttt{alexander.zuyev@gmail.com} \\
  \And
 Ju.~Kalosha \\
  Institute of Applied Mathematics and Mechanics, Sloviansk, Ukraine, \\
  \texttt{julykucher@gmail.com} \\
}

\begin{document}
\include{Macro}
\maketitle

\begin{abstract}
    This paper is devoted to analyzing the observer convergence rate for a class of linear control systems in a Hilbert space. To characterize the polynomial stability of the observer error system, we apply the spectral theory of linear operators and explicitly construct the resolvent of the corresponding infinitesimal generator.
    The asymptotic behavior of the resolvent on the imaginary axis is studied to describe the rate of decay of the observation error.
    The estimated decay rate is illustrated through an example of an oscillating flexible structure with one-dimensional output.
\end{abstract}

\keywords{infinite-dimensional control system, observer, polynomial decay rate, spectral analysis, Riesz basis.}

{\bf MSC:} 93B53, 93C25, 93B60, 93D99.

\section{Introduction}

The present work is a continuation of the study of the observation problem for linear dynamical systems in infinite-dimensional spaces.
Once the observer is constructed for an input-output system, the convergence of the observation error becomes particularly important.
Even when the asymptotic stability of the error dynamics is established, it remains important to determine whether a desired rate of convergence can be achieved.
One of the major challenges in investigating infinite-dimensional systems is that asymptotic stability does not imply exponential stability. In particular, the rate of convergence of the solution may be polynomial. Corresponding examples of such systems can be found in~\cite{BPS2019,BH2007,CQMS2022,D2007,Fu2008,Phung2007,RZZ2005}.

In the present work, we focus on investigating the convergence rate of previously derived observation error dynamics in Hilbert space. Specifically, we establish sufficient conditions for the polynomial stability of the zero error.
It is well-known that the decay rate of solutions to linear infinite-dimensional systems is related to the growth rate of the resolvent of the infinitesimal generator along the imaginary axis of the complex plane. Studies on the resolvent behavior and related solution properties can be found, for instance, in~\cite{BD2008,Lebeau1996,Tomilov2001}.
It is noted n~\cite{BEPS2006} that the polynomial stability of a semigroup implies that the spectrum of generator
lies in the open left half-plane and that  its resolvent satisfies a polynomial growth estimate.
In the aaforementioned  work, the polynomial decay of the solutions  is characterized for linear systems with commuting normal operators in Banach space, and the results are applied to coupled systems such as conservative and damped wave equations, as well as wave and plate equations.

It is worth mentioning the work~\cite{Liu2005}, which studies the rate of decay of solutions to linear systems whose equilibrium state is stable but not necessarily exponentially stable.
The authors consider equations for which the resolvent of the semigroup generator is an unbounded operator along the imaginary axis, and provide a characterization of the order of growth of the resolvent.
Another result, based on the growth of the resolvent on the imaginary axis, is introduced in~\cite{LR2007}, where a polynomial energy decay rate is obtained for a system of wave equations.

The work~\cite{BT2010} summarizes a general approach of determining the polynomial convergence rate of solutions via analysing the rate of growth of the resolvent's norm.
Specifically, the authors of~\cite{BT2010} established the equivalence of the polynomial stability property and power asymptotics of the resolvent norm along the imaginary axis. For further results in this area, see, e.g.,~\cite{BCT2016,BCPQ2024,CST2020,DRV2024,R2023,RSS2019}.

Recall that the trivial solution of a system $\dot x(t) = Ax(t)$ in a Hilbert space $H$ is polynomially stable if there exist positive constants $\varkappa$ and $\eta$, such that
$$\|x(t)\| \leq \varkappa\, t^{-\eta} \|Ax(0)\|$$
for $t>0$, $x(0) \in D(A)\subset H$ (cf.~\cite{BEPS2006}).
We will use this concept for describing the convergence rate of the observer constructed in~\cite{ZK2023ECC}.

In current work, we prove the polynomial convergence of the error estimate for the Luenberger-type observer constructed in~\cite{ZK2023ECC}.
The main result is formulated in Section~\ref{sec:stre} and its proof is presented in Section~\ref{sec:normest}.

The structure of the rest of this paper is as follows.
In Section~\ref{sec:stre}, we introduce the basic notations and formulate the main result.
Section~\ref{sec:spcnl} provides an analysis of the spectrum of the infinitesimal generator for the considered error dynamics.
In Section~\ref{sec:res}, we construct the resolvent for the system with one-dimensional output.
Section~\ref{sec:chareq} provides a detailed analysis of the characteristic equation.
    It is shown that every eigenvalue is located in some neighborhood of a purely imaginary complex number
    corresponding to an eigenmode of the system.
In Section~\ref{sec:Riesz}, we establish the Riesz basis property for the considered class of systems and derive the transformation of the infinitesimal generator to a diagonal form.
    This transformation is utilized in the proof of the main result in
Section~\ref{sec:normest}.

In Section~\ref{sec:ex}, we propose an illustrative example of the considered class of systems and demonstrate the spectrum location.
The Appendix contains proofs of some auxiliary statements.

\section{Problem statement and main result}\label{sec:stre}

In the paper~\cite{ZK2023ECC}, an observer design was proposed for a class of linear dynamical systems that describe the controlled oscillations of flexible structures. It has been proven that the observer's error dynamics is asymptotically stable under natural observability assumptions.
The goal of the present paper is to establish conditions for the polynomial convergence of this type of observer. To achieve this, we will construct the resolvent of the semigroup generator and analyze its asymptotic behavior along the imaginary axis.
Based on this resolvent analysis, we will characterize the decay rate of the observation error under additional assumptions regarding the parameters of the control system.

We consider the following linear control system with output:
\begin{equation}\label{eq:linsys_fd}
	\dot z = {\tilde A}z + Bu, \quad z\in H=\ell^2\times \ell^2,\quad u\in\mathbb R^{k+1},\quad k\in \mathbb N,
\end{equation}
\begin{equation}\label{eq:output_fd}
	y = Cz, \quad y\in\mathbb R,
\end{equation}
where the state vector
$z=\begin{pmatrix} \xi \\ \eta \end{pmatrix}$
is defined through components ${\xi=(\xi_1,\xi_2,\dots)^\top\in \ell^2}$ and ${\eta=(\eta_1,\eta_2,\dots)^\top\in \ell^2}$,
the control $u=(u_0,u_1,\dots,u_k)^\top$ is a finite-dimensional vector.
The operators ${\tilde A}:D({\tilde A})\to H$, \linebreak
$B:\mathbb R^{k+1}\to H$,\; $C:H\to \mathbb R$ are defined as follows:
$D({\tilde A})=\left\{z=\begin{pmatrix} \xi \\ \eta \end{pmatrix}\in H:\sumf{j}\omega_j^2(\xi_j^2+\eta_j^2)<\infty\right\}$, \linebreak
${\tilde A}: z = \begin{pmatrix} \xi \\ \eta \end{pmatrix} \mapsto {\tilde A}z = \begin{pmatrix} \Omega \eta \\ -\Omega \xi \end{pmatrix}$,\;
$Bu = \begin{pmatrix} 0 \\ B_1 u \end{pmatrix}$,\; $C z = C_1 \xi$,
where
$\Omega = {\rm diag}(\omega_1,\omega_2,\dots)$,\; $B_1:\mathbb R^{k+1}\to \ell^2$,
$B_1 = \begin{pmatrix} b_{10} & \ldots & b_{1k} \\ b_{20} & \ldots & b_{2k} \\ \vdots & \vdots & \vdots \end{pmatrix}$,\;
$C_1:\ell^2\to \mathbb R$,\; $C_1 \xi = \sum\limits_{j=1}^\infty c_j\, \xi_j$.

The Luenberger-type observer for the system~\eqref{eq:linsys_fd} and~\eqref{eq:output_fd} was proposed in~\cite{ZK2023ECC} in the following form:
\begin{equation}\label{eq:observer_fd}
	\dot{\tilde z}(t) = ({\tilde A}-FC)\tilde z(t) + Bu(t) + Fy(t),
\end{equation}
where $F = \begin{pmatrix} f \\ g \end{pmatrix}:\mathbb R\to H$ is defined through its components $f,g:\mathbb R\to\ell^2$, in the form
$f = (f_1, f_2, \dots)^\top$, $g = (g_1, g_2, \dots)^\top$
with $f_j = \gamma c_j$,\; $g_j = 0$,\; $j=1,2,\dots$, and $\gamma>0$ being an arbitrary gain parameter.

In the cited work~\cite{ZK2023ECC}, conditions for observer convergence were investigated.
For this purpose, we analyzed the observation error dynamics and proved its asymptotic stability using a modification of Lyapunov's direct method (cf.~\cite{zuyev2003partial,zuyev2005stabilization,zuev2006partial}) and the invariance principle.
The focus of the current work is the estimation of the error decay rate.
This study is based on the spectral analysis of the infinitesimal generator of the error dynamics system.

Let $e(t) =  z(t) - \tilde z(t)$ be the observation error. Subtracting~\eqref{eq:observer_fd} from~\eqref{eq:linsys_fd}, we obtain the system
\begin{equation*}
	\dot e(t) = \hat A \,e(t),
\end{equation*}
where
$$
e(t) = \begin{pmatrix} \Delta(t) \\ \delta(t) \end{pmatrix}\in {\ell^2}\times {\ell^2},\;
\hat A = \begin{pmatrix} \breve A & \Omega \\ -\Omega & 0 \end{pmatrix},\;
\breve A = -\gamma
    \begin{pmatrix}
        c_1^2  & c_1c_2 & \ldots \\
        c_2c_1 & c_2^2  & \ldots \\
        \vdots & \vdots & \ddots
    \end{pmatrix}.
    $$
Introducing the new variables $q_j = \Delta_j + i\,\delta_j$,\; $p_j = \Delta_j - i\,\delta_j$, we derive the error dynamics in complex coordinates:
\begin{equation*}
    \begin{aligned}
        & \dot q_j + \dot p_j = -\gamma c_j \sumf{\iota} c_\iota (q_\iota + p_\iota) - i\,\omega_j (q_j - p_j), \\
        & \dot q_j - \dot p_j = - i\,\omega_j (q_j + p_j),
    \end{aligned} \qquad j\in\mathbb N,
\end{equation*}
or in the matrix form:
\begin{equation}\label{eq:err_dsys}
    \dot\varepsilon = A \varepsilon,
\end{equation}
where $\varepsilon = \begin{pmatrix} q \\ p \end{pmatrix}$,\; $q = (q_1, q_2, \ldots)^\top$,\; $p = (p_1, p_2, \ldots)^\top$,\;
$A = \begin{pmatrix} - i\,\Omega + \check A & \check A \\ \check A & i\,\Omega + \check A \end{pmatrix}$,\;
$\check A = \frac12 \breve A$.

The main result of this work is presented below.
\begin{theorem}\label{th:main}
    Let the following assumptions hold:
    \begin{itemize}
        \item[(A1)] The diagonal entries $\omega_j$ of $\Omega$ form an increasing sequence of distinct real numbers, i.e., $0<\omega_1<\omega_2<\ldots$, and $c_j\neq0$ for all $j\in{\mathbb N}$.
        \item[(A2)] There exists a constant $\kappa>0$ such that $\omega_{j+1}-\omega_j\geq\kappa$, $j=1,2,\ldots$\,.
        \item[(A3)] There exist constants $\alpha>0$, $\beta>0$, and $k_0 \in {\mathbb N}$, such that $|c_k|\,\omega_k^{\alpha/2} \geq \beta$ for all $k \geq k_0$.
    \end{itemize}
    Then every solution $\varepsilon(t)$,\: $t\ge 0$,\: of system~\eqref{eq:err_dsys} with $\varepsilon(0)\in D(A)$ satisfies the estimate
    \begin{equation}\label{eq:SolEst}
        \|\varepsilon(t)\| \leq \tilde\beta (t+1)^{-1/\alpha} \|A \varepsilon(0)\|\quad\text{for all}\; t\ge 0,
    \end{equation}
    with $\tilde\beta \geq \beta$.
\end{theorem}
The proof of Theorem~\ref{th:main} will be provided in Section~4, based on the auxiliary results presented in Section~3.

\section{Spectral analysis of the infinitesimal generator}\label{sec:spcnl}
\subsection{Construction of the the resolvent in complex Hilbert space}\label{sec:res}

Let us construct the resolvent of the operator $A$. For that we will solve the equation
\begin{equation}\label{eq:res_eq}
    (A-\lambda\rm{I})\varepsilon = \hat\varepsilon
\end{equation}
with respect to $\varepsilon$ for any $\hat\varepsilon\in H$ for all $\lambda\in\mathbb C\setminus\{0\}$. Here and further, the symbol $\rm{I}$ denotes the identity operator of the corresponding dimension.

We will show that the resolvent of $A$ is well-defined for such $\lambda \in\mathbb C \setminus\{\pm\,i\,\omega_j\}$ that
\begin{equation}\label{eq:res_cs0}
    \gamma\lambda \sumf{j}\frac{c_j^2}\qdn + 1 \neq 0,
\end{equation}
and for $\lambda=\pm i\,\omega_k$, $k\in{\mathbb N}$.

First, let $\lambda\neq\pm\,i\,\omega_j$, $j\in{\mathbb N}$. Denote $\phi(q,p) = \sum\limits_{\iota=1}^\infty c_\iota (q_\iota + p_\iota)$. Equation~\eqref{eq:res_eq} can be represented as the following system:
\begin{equation*}
    \begin{aligned}
        -(i\,\omega_j + \lambda)\,q_j & = \frac\gamma2\,c_j\,\phi(q,p) + \hat q_j, \\
         (i\,\omega_j - \lambda)\,p_j & = \frac\gamma2\,c_j\,\phi(q,p) + \hat p_j,
    \end{aligned} \qquad j\in\mathbb N.
\end{equation*}
In order to obtain the coupling equation for the functional $\phi$, we add the above two equations, multiply the result by $c_j$ and sum over $j=1,2,\ldots$, and come up with the following equation:
\begin{equation*}
    \left(1 + \gamma\lambda \sumf{j}\frac{c_j^2}\qdn\right)\phi = \sumf{j}\frac{c_j}\qdn\,\bigl((i\,\omega_j-\lambda)\hat q_j - (i\,\omega_j+\lambda)\hat p_j\bigr).
\end{equation*}
Provided by conditions~\eqref{eq:res_cs0} and $\lambda\neq\pm\,i\,\omega_j$, we can express $\phi$ as follows:
\begin{equation}\label{eq:res_cs0_phi}
    \phi = \frac{\sumf{j}\dfrac{c_j}\qdn\,\bigl((i\,\omega_j-\lambda)\hat q_j - (i\,\omega_j+\lambda)\hat p_j\bigr)}{1+\gamma\lambda\sumf{j}\dfrac{c_j^2}\qdn}.
\end{equation}
Then the resolvent of the operator $A$ can be represented as the following operator:
$$R(\lambda;A) = (A-\lambda\rm{I})^{-1}: \hat\varepsilon = \begin{pmatrix} \hat q \\ \hat p \end{pmatrix} \mapsto R \hat\varepsilon = \begin{pmatrix} q\,(\hat q,\hat p) \\ p\,(\hat q,\hat p) \end{pmatrix}$$
with
\begin{equation*}
    q_j = -\frac1{i\,\omega_j+\lambda}\,\left(\frac\gamma2\, c_j\phi + \hat q_j\right), \quad
    p_j =  \frac1{i\,\omega_j-\lambda}\,\left(\frac\gamma2\, c_j\phi + \hat p_j\right)
\end{equation*}
and $\phi=\phi\,(\hat q,\hat p)$ given in~\eqref{eq:res_cs0_phi}.

Now let us consider the cases $\lambda=\pm\,i\,\omega_j$ for some $j=k$.

\textbf{(A)} $\lambda=i\,\omega_k$. \\
In this case, equation~\eqref{eq:res_eq} can be represented as system
\begin{align*}
    & -\frac12\gamma c_k \sumf{\iota} c_\iota(q_\iota+p_\iota) = \hat p_k, \\ 
    & -\frac12\gamma c_j \sumf{\iota} c_\iota(q_\iota+p_\iota) - (i\omega_j+\lambda)\,q_j = \hat q_j, \; j\in\mathbb N, \\ 
    & -\frac12\gamma c_j \sumf{\iota} c_\iota(q_\iota+p_\iota) + (i\omega_j-\lambda)\,p_j = \hat p_j, \; j\in\mathbb N\setminus\{k\}, 
\end{align*}
from which the resolvent $R(\lambda;A)$ for $\lambda=i\,\omega_k$ is defined as
$R: \: \hat\varepsilon = \begin{pmatrix} \hat q \\ \hat p \end{pmatrix} \mapsto R \hat\varepsilon = \begin{pmatrix} q\,(\hat q,\hat p) \\ p\,(\hat q,\hat p) \end{pmatrix}$, where
\begin{align*}
    & q_j = -\frac1{i\omega_j+\lambda} \left(\hat q_j - \frac{c_j}{c_k}\, \hat p_k\right), \; j\in\mathbb N, \\
    & p_j = \frac1{i\omega_j-\lambda} \left(\hat p_j - \frac{c_j}{c_k}\, \hat p_k\right), \; j\in\mathbb N\setminus\{k\}, \\
    \text{and} \;\; & p_k = -\frac1{c_k}\left(\frac2{\gamma c_k}\hat p_k + \sumf{\iota} c_\iota q_\iota + \sum\limits_{\iota\neq k} c_\iota p_\iota\right).
\end{align*}

\textbf{(B)} $\lambda=-i\,\omega_k$. \\
In this instance, equation~\eqref{eq:res_eq} reads as
\begin{align*}
    & -\frac12\gamma c_k \sumf{\iota} c_\iota(q_\iota+p_\iota) = \hat q_k, \\ 
    & -\frac12\gamma c_j \sumf{\iota} c_\iota(q_\iota+p_\iota) - (i\omega_j+\lambda)\,q_j = \hat q_j, \; j\in\mathbb N\setminus\{k\}, \\ 
    & -\frac12\gamma c_j \sumf{\iota} c_\iota(q_\iota+p_\iota) + (i\omega_j-\lambda)\,p_j = \hat p_j, \; j\in\mathbb N. 
\end{align*}
The resolvent $R(\lambda;A)$ for $\lambda=-i\,\omega_k$ is then defined as
$R: \: \varepsilon = \begin{pmatrix} \hat q \\ \hat p \end{pmatrix} \mapsto R \hat\varepsilon = \begin{pmatrix} q\,(\hat q,\hat p) \\ p\,(\hat q,\hat p) \end{pmatrix}$, where
\begin{align*}
    & q_j = -\frac1{i\omega_j+\lambda} \left(\hat q_j - \frac{c_j}{c_k}\hat q_k\right), \; j\in\mathbb N\setminus\{k\}, \\
    & p_j = \frac1{i\omega_j-\lambda} \left(\hat p_j - \frac{c_j}{c_k}\, \hat q_k\right), \; j\in\mathbb N, \\
    & q_k = -\frac1{c_k} \left(\frac2{\gamma c_k}\, \hat q_k + \sumf{\iota} c_\iota p_\iota + \sum\limits_{\iota\neq k} c_\iota q_\iota\right).
\end{align*}


Thus, the resolvent of the operator $A$ is constructed for the system with one-dimensional output.

\subsection{Analysis of the characteristic equation}\label{sec:chareq}
Let us rewrite the characteristic equation
$\sumf{j}\dfrac{c_j^2}\qdn + \dfrac1{\gamma\lambda} = 0$
in the form
\begin{equation}\label{eq:chareq}
    f(\lambda) := \sumf{j} \frac{c_j^2}{\omega_j} \left(\frac1{\lambda-i\omega_j} - \frac1{\lambda+i\omega_j}\right) + \frac{2i}{\gamma\lambda} = 0.
\end{equation}
Furthermore, we will show that there are no eigenvalues in ${\mathbb C}\setminus \bigcup\limits_{j\in{\mathbb N}}B(\pm i\omega_j; \breve R_\lambda)$, where $B(\pm i\omega_j; \breve R_\lambda)$ are the closed disks centered at $\pm i\omega_j$ with radii
\begin{equation}\label{Rlambda}
    \breve R_\lambda = \dfrac\gamma2 |\lambda|\,\sumf{j}\dfrac{c_j^2}{\omega_j}.
\end{equation}
The properties of the spectrum of $A$ are summarized below.

\begin{proposition}\label{prp:nlfl}
    The spectrum $\sigma(A)$ of the operator $A$ is discrete and satisfies the following properties:
\begin{enumerate}
    \item $\sigma(A) \,\cap\, i{\mathbb R} = \emptyset$;
    \item the set of eigenvalues of $A$ is symmetric with respect to the real line;
    \item $\inf\limits_{j\in{\mathbb N}} |\lambda\pm i\omega_j| \leq \breve R_\lambda$ for each eigenvalue $\lambda$,
    where $\breve R$ is defined in~\eqref{Rlambda}.
\end{enumerate}
    \end{proposition}

The proof is presented in the Appendix. We refer the reader to Section~\ref{sec:ex} for an illustration of the eigenvalue distribution.

Consider the function
\begin{equation*}
    \begin{aligned}
        F(\lambda)
        & = (\lambda-i\omega_k) f(\lambda) \\
        & = (\lambda-i\omega_k) \sumk \frac{c_j^2}{\omega_j} \left(\frac1{\lambda-i\omega_j} - \frac1{\lambda+i\omega_j}\right) + (\lambda-i\omega_k) \frac{2i}{\gamma\lambda} + \frac{2i c_k^2}{\lambda+i\omega_k},
    \end{aligned}
\end{equation*}
which is analytic in a neighbourhood of $i\omega_k$ and thus admits the power series expansion at~$\lambda=i\omega_k$:

\begin{equation*}
    F(\lambda) = F(i\omega_k) + (\lambda-i\omega_k)\,F'(i\omega_k) + r(\lambda),
\end{equation*}
where $r(\lambda) = \sum\limits_{n=2}^\infty \dfrac{F^{(n)}(i\omega_k)}{n!} (\lambda-i\omega_k)^n$ is the remainder term.

We denote the linear part of the above power series by $g(\lambda) = F(i\omega_k) + (\lambda-i\omega_k)\,F'(i\omega_k)$ and calculate its root
\begin{equation*}
    \begin{aligned}
        \lambda_k^*
        & = -\frac{F(i\omega_k)}{F'(i\omega_k)} + i\omega_k \\
        & = -\frac{c_k^2}{\omega_k} \left[2i\sumk \frac{c_j^2}{\omega_j^2-\omega_k^2} + \frac{ic_k^2}{2\omega_k^2} + \frac2{\gamma\omega_k}\right]^{-1} + i\omega_k.
    \end{aligned}
\end{equation*}
Note that ${\rm Re}\,\lambda_k^*<0$.

We assume that $\omega_{k+1}-\omega_k > \omega_k-\omega_{k-1} > 0$ and denote by $R_0$ the radius of convergence of the above power series,
$R_0 = \min\limits_{j\neq k}{\rm dist}(i\omega_k,i\omega_j) = \omega_k-\omega_{k-1}>0$.

The remainder $r(\lambda)$ is convergent series in the disk $\Gamma_0 = \{\lambda \in {\mathbb C}: |\lambda-i\omega_k| < R_0\}$. It can be presented in the form
$$r(\lambda)=(\lambda-i\omega_k)^2\rho(\lambda),$$
where $\rho(\lambda)$ is analytical function in the disk $\Gamma_1 = \{\lambda \in {\mathbb C}: |\lambda-i\omega_k|\leq R_1\}$ with $R_1<R_0$. Then $\rho(\lambda)$ is bounded in $\Gamma_1$.


The following lemma shows that local behaviour of the characteristic function $f(\lambda)$ is determined by the linear function $g(\lambda)$.

\begin{lemma}\label{le:Rouche}
    Define an open disk $\Gamma = \{\lambda \in {\mathbb C}: |\lambda-\lambda_k^*| < R_k\}$.
    Assume that $\Gamma \subset \Gamma_1$.
    Let $M=\sup\limits_{|\lambda-i\omega_k|\leq R_1}|\rho(\lambda)|$ and
    \begin{equation}\label{eq:Mneq1}
        0< M <\dfrac{|F'(i\omega_k)|^2}{4|F(i\omega_k)|}.
    \end{equation}
    If $\sqrt{R_k} \in {\cal I} = \left(b-\sqrt{b^2-c}\:;\: b+\sqrt{b^2-c}\,\right)$, where $b=\sqrt{\dfrac{|F'(i\omega_k)|}{4M}}$,\; $c=|\lambda_k^*-i\omega_k|=\left|\dfrac{F(i\omega_k)}{F'(i\omega_k)}\right|$,
    then the function $f(\lambda)$ has exactly one zero in the open neighborhood $\Gamma$.
\end{lemma}

\begin{figure}[thpb]
    \centering
    \includegraphics[scale=.8]{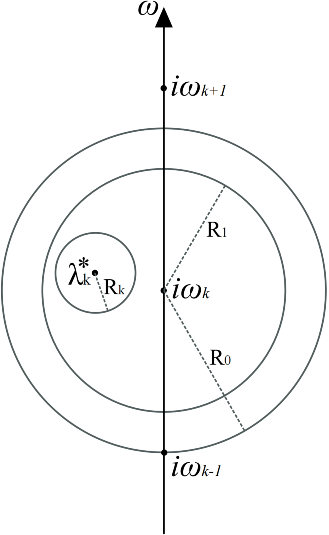}
    \caption{Schematic representation of the eigenvalues location}
	\label{fig:Eigloc}
\end{figure}

\begin{proof}
    First we prove that $|r(\lambda)| < |g(\lambda)|$ for all $\lambda$ on the contour $\partial\Gamma=\{\lambda\in{\mathbb C}: |\lambda-\lambda_k^*|=R_k\}$.

    Since $g(\lambda_k^*)=0$, the linear function $g(\lambda)$ can be presented as
    $$g(\lambda) = g'(\lambda_k^*)(\lambda-\lambda_k^*) = F'(i\omega_k)(\lambda-\lambda_k^*),$$
    so $|g(\lambda)|=|F'(i\omega_k)|\,R_k$ on the contour $\partial\Gamma$.

    For $\lambda$ in the disk $\Gamma_1$ we have
    $|r(\lambda)| \leq |\lambda-i\omega_k|^2\,M.$
    Thus, the desired inequality takes the form
    \begin{equation}\label{eq:Rneq1}
        |\lambda-i\omega_k| < \sqrt{\frac{R_k}M |F'(i\omega_k)|}.
    \end{equation}
    By the triangle inequality applied on $\partial\Gamma$,
    $$|\lambda-i\omega_k| \leq R_k + |\lambda_k^*-i\omega_k| = R_k + \left|\frac{F(i\omega_k)}{F'(i\omega_k)}\right|.$$
    If constants $R$ and $M$ fulfill conditions of the lemma, the following inequality is satisfied:
    $$R_k + \left|\frac{F(i\omega_k)}{F'(i\omega_k)}\right| < \sqrt{\frac{R_k}M\,|F'(i\omega_k)|},$$
    which implies~\eqref{eq:Rneq1}.

    Now, the conditions of the Rouch{\'e} theorem are satisfied for functions $F(\lambda)$ and $g(\lambda)$, so these functions have the same number of zeros in the disk~$\Gamma$. Since $\lambda-i\omega_k \neq 0$ in this disk, we conclude that the characteristic function $f(\lambda)$ has exactly one zero in~$\Gamma$.
\end{proof}

Thus we have shown that under assumptions of Lemma~\ref{le:Rouche} eigenvalue $\lambda_k$ lies inside the neighborhood $\Gamma$ of the value $\lambda_k^*$.
Further we investigate the size of circle $\Gamma$, namely, under what assumptions its radius $R_k$ does not exceed $\Theta|{\rm Re}\lambda_k^*|$, $\Theta\in(0;1)$, so that all eigenvalues are separated from the imaginary axis.
For simplicity and without loss of generality, let us consider $\Theta=\frac12$.

First we will show that inequality $R_k \leq \frac12|{\rm Re}\lambda_k^*|$ is not guaranteed by $\sqrt{R_k} \in {\cal I}$ because $\frac12|{\rm Re}\lambda_k^*| = \dfrac{c_k^2}{\gamma \omega_k^2\,|F'(i\omega_k)|^2}$ does not exceed $\left(b+\sqrt{b^2-c}\right)^2$. In fact, the following proposition states even stronger inequality.

\begin{proposition}\label{prp:spr}
    Let $M$ and $b$ satisfy assumptions of Lemma~\ref{le:Rouche}. Then
    $\frac12|{\rm Re}\lambda_k^*| < b^2$.
\end{proposition}

\begin{lemma}\label{le:R}
    Let constants $M$, $b$ and $c$ be defined as in Lemma~\ref{le:Rouche}.
    If
    \begin{equation}\label{eq:Mneq2}
        M < \dfrac{\gamma \omega_k |F'(i\omega_k)|^3}{c_k^2 (\gamma \omega_k |F'(i\omega_k)| + 1)^2},
    \end{equation}
    then $R_k \leq \frac12|{\rm Re}\lambda_k^*|$ for all $\omega_k > 1$.
\end{lemma}

From the considerations given above, all eigenvalues of the operator $A$ are distinct complex numbers provided by assumption {\it (A2)}.
Since each eigenvalue $\lambda_n$ is simple root of~\eqref{eq:chareq}, examination of the spectral projection of the operator $A$, corresponding to $\lambda_n$, guarantees that $\lambda_n$ is algebraically simple.

\begin{proposition}\label{prp:O}
    Let assumptions of Lemma~\ref{le:R} be satisfied. Then $\lambda_k \underset{k\to\infty}= i\omega_k + {\cal O} \left(c_k^2 \right)$, $\bar\lambda_k \underset{k\to\infty}= -i\omega_k + {\cal O} \left(c_k^2 \right)$ for all eigenvalues $\lambda_k$ and $\bar\lambda_k$ of the operator $A$.
\end{proposition}

The proofs of Propositions~\ref{prp:spr} and~\ref{prp:O} and Lemma~\ref{le:R} are given in the Appendix.

\subsection{Riesz basis property}\label{sec:Riesz}
For each $\lambda$ satisfying characteristic equation~\eqref{eq:chareq} find nonzero vector
$\varepsilon = \begin{pmatrix} q \\ p \end{pmatrix}$, such that
$(A - \lambda I) \varepsilon = 0$, which is written as
\begin{equation*}
    \begin{aligned}
        -(i\,\omega_j + \lambda)\,q_j & = \frac\gamma2\,c_j\,\phi(q,p), \\
         (i\,\omega_j - \lambda)\,p_j & = \frac\gamma2\,c_j\,\phi(q,p),
    \end{aligned} \qquad j\in\mathbb N,
\end{equation*}
where $\phi(q,p) = \sum\limits_{\iota=1}^\infty c_\iota (q_\iota + p_\iota)$.
The solution (recall that $\lambda \neq \pm i\omega_j$)
\begin{equation*}
    q_j = -\frac1{i\omega_j + \lambda}\, \frac\gamma2\, c_j\, \phi, \quad
    p_j = \frac1{i\omega_j - \lambda}\, \frac\gamma2\, c_j\, \phi
\end{equation*}
leads to the coupling equation for the parameter $\phi$:
$$\left(1 - \frac\gamma2 \sumf{j} c_j \left( \frac1{i\omega_j - \lambda} - \frac1{i\omega_j + \lambda} \right) \right) \phi = 0.$$
For all eigenvalues $\lambda$ expression in brackets is equal to zero, and the last equation is fulfilled for any $\phi$.
We choose $\phi = \frac2\gamma$ and obtain for every eigenvalue $\lambda$ the corresponding eigenvector $\varepsilon$ with components
$$q_j = -\frac{c_j}{i\omega_j + \lambda}, \quad
    p_j = \frac{c_j}{i\omega_j - \lambda}.$$

For clarity we agree to denote the eigenvalue located in the neighborhood of $i\omega_n$ by $\lambda_n$, then $\lambda_{-n} = \bar\lambda_n$ would be the one in the neighborhood of $-i\omega_n$.
Let us multiply each eigenvector $\varepsilon_n$ associated with $\lambda_n$ by $\dfrac{i\omega_n-\lambda_n}{c_n}$,
and multiply each eigenvector $\varepsilon_{-n}$ associated with $\bar\lambda_n$ by\; $-\dfrac{i\omega_n+\bar\lambda_n}{c_n}$.

\begin{lemma}\label{le:Riesz}
    Let assumption~{\it (A2)} be satisfied.
    Then the eigenvectors $\{\varepsilon_n\}_{n=1}^\infty$ and $\{\varepsilon_{-n}\}_{n=1}^\infty$ of $A$ form a Riesz basis in the space~$H$.
\end{lemma}

\begin{proof}
    Unperturbed operator ($\gamma=0$) $A_0 = \begin{pmatrix} - i\,\Omega & 0 \\ 0 & i\,\Omega \end{pmatrix}$ has eigenvalues $\nu = \pm i\omega_j$.
    For $\nu_{-n} = \bar\nu_n = -i\omega_n$ the components of the corresponding eigenvector
    $\tilde\varepsilon_{-n} = \begin{pmatrix} \tilde q \\ \tilde p \end{pmatrix}_{\!\!-n}$ are
    $\tilde q_{-n} = 1$, $\tilde q_\iota = 0$, $\forall \iota \neq -n$, $\tilde p_j = 0$, $\forall j\in{\mathbb N}$;
    components of the eigenvector corresponding to eigenvalue $\nu_n = i\omega_n$ are
    $\hat q_j = 0$, $\forall j\in{\mathbb N}$, $\hat p_n = 1$, $\hat p_\iota = 0$, $\forall \iota \neq n$.

    Further we prove that the system of eigenvectors $\{\varepsilon_n\}_{n=1}^\infty \cup \{\varepsilon_{-n}\}_{n=1}^\infty$ is quadratically close
    to the orthonormal basis $\{\tilde\varepsilon_n\}_{n=1}^\infty \cup \{\tilde\varepsilon_{-n}\}_{n=1}^\infty$.
    Consider
    \begin{equation*}
        \begin{aligned}
            \| \varepsilon_n - \tilde\varepsilon_n \|_H^2 =
                & \sum\limits_{j\neq n} \frac{c_j^2}{c_n^2}\, |i\omega_n-\lambda_n|^2 \left( \frac1{|i\omega_j+\lambda_n|^2} + \frac1{|i\omega_j-\lambda_n|^2} \right)
                    + \left|\frac{i\omega_n-\lambda_n}{i\omega_n+\lambda_n}\right|^2 \\
            & + \sum\limits_{j\neq n} \frac{c_j^2}{c_n^2}\, |i\omega_n+\bar\lambda_n|^2 \left( \frac1{|i\omega_j+\bar\lambda_n|^2} + \frac1{|i\omega_j-\bar\lambda_n|^2} \right)
                + \left|\frac{i\omega_n+\bar\lambda_n}{i\omega_n-\bar\lambda_n}\right|^2 \\
            \leq & \sum\limits_{j\neq n} \frac{c_j^2}{c_n^2}\, |i\omega_n-\lambda_n|^2 \left( \frac1{|i\omega_1+\lambda_n|^2} + \frac1{|i\omega_{n-1}-\lambda_n|^2} + \frac1{|i\omega_{n+1}-\lambda_n|^2} \right) \\
            & + \sum\limits_{j\neq n} \frac{c_j^2}{c_n^2}\, |i\omega_n+\bar\lambda_n|^2 \left( \frac1{|i\omega_{n-1}+\bar\lambda_n|^2} + \frac1{|i\omega_{n+1}+\bar\lambda_n|^2} + \frac1{|i\omega_1-\bar\lambda_n|^2} \right) \\
            & + \left|\frac{i\omega_n-\lambda_n}{i\omega_n+\lambda_n}\right|^2 + \left|\frac{i\omega_n+\bar\lambda_n}{i\omega_n-\bar\lambda_n}\right|^2.
        \end{aligned}
    \end{equation*}
    According to Proposition~\ref{prp:O}, $\dfrac{|i\omega_n-\lambda_n|^2}{c_n^2} \underset{n\to\infty}= {\cal O} \left(c_n^2 \right)$ and $\dfrac{|i\omega_n+\bar\lambda_n|^2}{c_n^2} \underset{n\to\infty}= {\cal O} \left(c_n^2 \right)$;
    sequences $\left\{ \dfrac1{|i\omega_{n-1}-\lambda_n|^2} \right\}_{n=1}^\infty$ and $\left\{ \dfrac1{|i\omega_{n+1}-\lambda_n|^2} \right\}_{n=1}^\infty$ are bounded due to the assumption of the lemma, as are
    sequences $\left\{ \dfrac1{|i\omega_{n-1}+\bar\lambda_n|^2} \right\}_{n=1}^\infty$ and $\left\{ \dfrac1{|i\omega_{n+1}+\bar\lambda_n|^2} \right\}_{n=1}^\infty$.

    Thus,
    $\sumf{n}\|\varepsilon_n - \tilde\varepsilon_n\|^2 < \infty$
    and $\sumf{n}\|\varepsilon_{-n} - \tilde\varepsilon_{-n}\|^2 < \infty$.
    The statement of the lemma then follows directly from~\cite[Theorem 2.38]{GW2019}.
\end{proof}

\begin{corollary}\label{crl:D}
    The operator $A$ is diagonalizable with spectral operator of the form 
    $G: \varepsilon = \begin{pmatrix} q \\ p \end{pmatrix} \mapsto  G\varepsilon = \begin{pmatrix} G_1\,q \\ G_2\,p \end{pmatrix}$,
    where
    $G_1 = {\rm diag}(\bar\lambda_1, \bar\lambda_2, \ldots)$, $G_2 = {\rm diag}(\lambda_1, \lambda_2, \ldots)$ with $\lambda_j$ and $\bar\lambda_j$~--- eigenvalues of the operator~$A$, $j\in{\mathbb N}$.
\end{corollary}

\section{Proof of the main result}\label{sec:normest}

In this section, we investigate the asymptotics of the resolvent $R(\lambda;G)$ for the diagonal operator $G$, obtained in Corollary~\ref{crl:D}, and apply these results to the proof of Theorem~\ref{th:main}.
As in the previous section, we denote by $\lambda_j$ the eigenvalues of $G$ located in the upper half of the complex plane (in a neighborhood of $i\omega_j$) and by $\bar\lambda_j$ the  eigenvalues located in the lower half of the complex plane  (in a neighborhood of $-i\omega_j$).

\begin{lemma}\label{le:ResNmBds}
    Let conditions of Lemma~\ref{le:R} be satisfied. If assumptions~{\it (A2)} and~{\it (A3)} are fulfilled, then
    $\|R(is;G)\| \underset{s\to\infty}= {\cal O}(|s|^\alpha)$ .
\end{lemma}

\begin{proof}
    Applying the H{\"o}lder inequality one has
    \begin{equation*}
        \begin{aligned}
            \left\|R(\lambda;G) \begin{pmatrix} q \\ p \end{pmatrix} \right\|^2
            & = \sumf{j} \dfrac{|q_j|^2}{|\bar\lambda_j-\lambda|^2} + \sumf{j} \dfrac{|p_j|^2}{|\lambda_j-\lambda|^2} \\
            & \leq \|q\|_{\ell^2}^2\, \sup\limits_{j\in{\mathbb N}} \frac1{|\bar\lambda_j - \lambda|^2} + \|p\|_{\ell^2}^2\, \sup\limits_{j\in{\mathbb N}} \frac1{|\lambda_j - \lambda|^2},
        \end{aligned}
    \end{equation*}
    where $\lambda_j$ and $\bar\lambda_j$ denote eigenvalues of~$G$.
    Thus,
    $$\|R(\lambda;G)\|_H^2 \leq \sup\limits_{j\in{\mathbb N}} \dfrac1{|\bar\lambda_j - \lambda|^2} + \sup\limits_{j\in{\mathbb N}} \dfrac1{|\lambda_j - \lambda|^2}.$$

    Assume that $\lambda=is$, $s\in{\mathbb R}\setminus\{0\}$.
    Let us cover the imaginary axis (excluding some neighborhood of zero\footnote{Since we are interested in the resolvent's norm asymptotic behavior as $|s|\to\infty$, we can strip any neighborhood of zero out of the considerations without loss of generality.}) by the set of segments of the form $I_k=\hat I_k\cup(-\hat I_k)$, where $\hat I_k=\left[\frac i2(\omega_{k-1}+\omega_k);\: \frac i2(\omega_k+\omega_{k+1}) \right]$, $k=2,3,\ldots$ (see Fig.~\ref{fig:Eigsg}), and estimate the supremum of $\|R(\lambda;G)\|_H$ over the segment $I_k$.

    \begin{figure}[thpb]
        \centering
        \includegraphics[scale=.8]{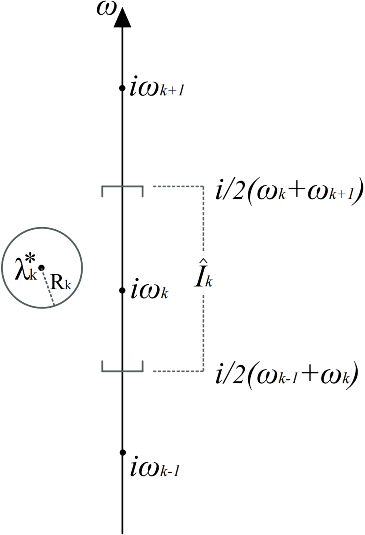}
        \caption{Schematic representation of the imaginary axis splitting}
    	\label{fig:Eigsg}
    \end{figure}

    \begin{equation*}
        \begin{aligned}
            \sup\limits_{\lambda\in I_k} \|R(\lambda;G)\|^2
            \leq\: & \sup\limits_{\substack{\lambda\in I_k \\ j<k}} \frac1{|\bar\lambda_j-\lambda|^2} + \sup\limits_{\substack{\lambda\in I_k \\ j>k}} \frac1{|\bar\lambda_j-\lambda|^2} + \sup\limits_{\lambda\in I_k} \frac1{|\bar\lambda_k-\lambda|^2} \\
                   & + \sup\limits_{\substack{\lambda\in I_k \\ j<k}} \frac1{|\lambda_j-\lambda|^2} + \sup\limits_{\substack{\lambda\in I_k \\ j>k}} \frac1{|\lambda_j-\lambda|^2} + \sup\limits_{\lambda\in I_k} \frac1{|\lambda_k-\lambda|^2} \\
               =\: & \frac1{|\bar\lambda_{k-1}+i/2(\omega_{k-1}+\omega_k)|^2} + \frac1{|\bar\lambda_{k+1}+i/2(\omega_{k+1}+\omega_k)|^2} + \frac1{|{\rm Re} \bar\lambda_k|^2} \\
                   & + \frac1{|\lambda_{k-1}-i/2(\omega_{k-1}+\omega_k)|^2} + \frac1{|\lambda_{k+1}-i/2(\omega_{k+1}+\omega_k)|^2} + \frac1{|{\rm Re} \lambda_k|^2} \\
            \leq\: & 2 \left( \frac1{|{\rm Re}\lambda_{k-1}|^2} + \frac1{|{\rm Re}\lambda_{k+1}|^2} + \frac1{|{\rm Re}\lambda_k|^2} \right).
        \end{aligned}
    \end{equation*}

    Note that $|{\rm Re}\lambda_k| \geq |{\rm Re}\lambda_k^*| - R_k$. In view of Lemma~\ref{le:R}, $|{\rm Re}\lambda_k| \geq \frac12 |{\rm Re}\lambda_k^*|$, $\forall k=1,2,\ldots$\,.
    Thus,
    $$\sup\limits_{\lambda\in I_k} \|R(\lambda;G)\| \leq \frac{2\sqrt6}{|{\rm Re}\lambda_{k+1}^*|}.$$
    Since $|{\rm Re}\lambda_k^*| \underset{k\to\infty}= {\cal O} \left(c_k^2\right)$ (see Proposition~\ref{prp:O}),
    \begin{equation}\label{eq:ORs}
        \sup\limits_{\lambda\in I_{k-1}} \|R(\lambda;G)\| \underset{\substack{k\to\infty \\ \lambda=is}}= {\cal O} \left(\dfrac1{c_k^2}\right).
    \end{equation}
    Obviously, $\dfrac1{c_k^2} \underset{k\to\infty}= {\cal O} \left(\omega_k^\alpha \right)$ provided by assumption~{\it(A3)}.

    Estimate~\eqref{eq:ORs} holds on any segment $I_{k-1}$ arbitrarily far from the origin, that implies asymptotic estimate
    $\|R(is;G)\| \underset{k\to\infty}= {\cal O}(\omega_k^\alpha) \underset{s\to\infty}= {\cal O}(|s|^\alpha)$.

\end{proof}


\begin{proof}[Proof of Theorem~\ref{th:main}.]

    Consider linear operator $Q:H\to H$, whose columns are eigenvectors $\varepsilon_n$, $\varepsilon_{-n}$ of $A$, $n=1,2,\ldots$\,. By definition of a Riesz basis, $Q$ is a bounded bijective operator, such that
    $A = Q G Q^{-1}$. Denote $\|Q\|_H=\beta_1$, $\|Q^{-1}\|_H=\beta_2$.
    In system~\eqref{eq:err_dsys} introduce the change of coordinates by $\varepsilon = Q \upsilon$, $\upsilon \in H$, thus obtaining linear system
    $\dot\upsilon(t) = G \upsilon(t)$, where $G$ is diagonal operator from Corollary~\ref{crl:D}.
    For the latter system the estimate obtained in Lemma~\ref{le:ResNmBds} implies
    $\|e^{tG} G^{-1}\| = {\cal O}(t^{-1/\alpha})$, $t\to\infty$ due to~{\cite[Theorem 2.4]{BT2010}},
    which in turn (taking into account assumption {\it (A3)}) yields
    \begin{equation*}
        \|\upsilon(t)\| \leq \beta (1+t)^{-1/\alpha} \|G \upsilon(0)\|.
    \end{equation*}
    The last inequality evidently leads to
    \begin{equation*}
        \|Q^{-1} \varepsilon(t)\| \leq \beta (1+t)^{-1/\alpha} \|G Q^{-1} \varepsilon(0)\|,
    \end{equation*}
    so
    $$\|\varepsilon(t)\| \leq \beta_1 \beta (1+t)^{-1/\alpha} \|Q^{-1} A \varepsilon(0)\| \leq \beta_1 \beta_2 \beta (1+t)^{-1/\alpha} \|A \varepsilon(0)\|.$$
    With $\tilde\beta = \beta_1 \beta_2 > 0$, estimate~\eqref{eq:SolEst} follows for all $\varepsilon(0) \in D(A)$, $t \geq 0$.

\end{proof}

\section{An illustrative example}\label{sec:ex}

As an example, we consider equation~\eqref{eq:chareq} with $\omega_j = \vartheta j^2$ and $c_j = \varsigma/j$,\; where $\vartheta$ and $\varsigma$ are positive constants.
The quadratic growth of $\omega_j$ is typical for flexible beam models, see, e.g.~{\cite[Chapter~4]{LGM1999}},~\cite{KZB2021}.
It is easy to check that, in this case, Assumption~{\it(A2)} holds with $\kappa=3\vartheta$, and Assumption~{\it(A3)} is fulfilled with
$\alpha \geq \log_k\left(\beta/\varsigma\sqrt\vartheta\right)+1$ and $k_0=2$ for any $\beta>0$.

\begin{figure}[thp]
    \centering
    \includegraphics[scale=.5]{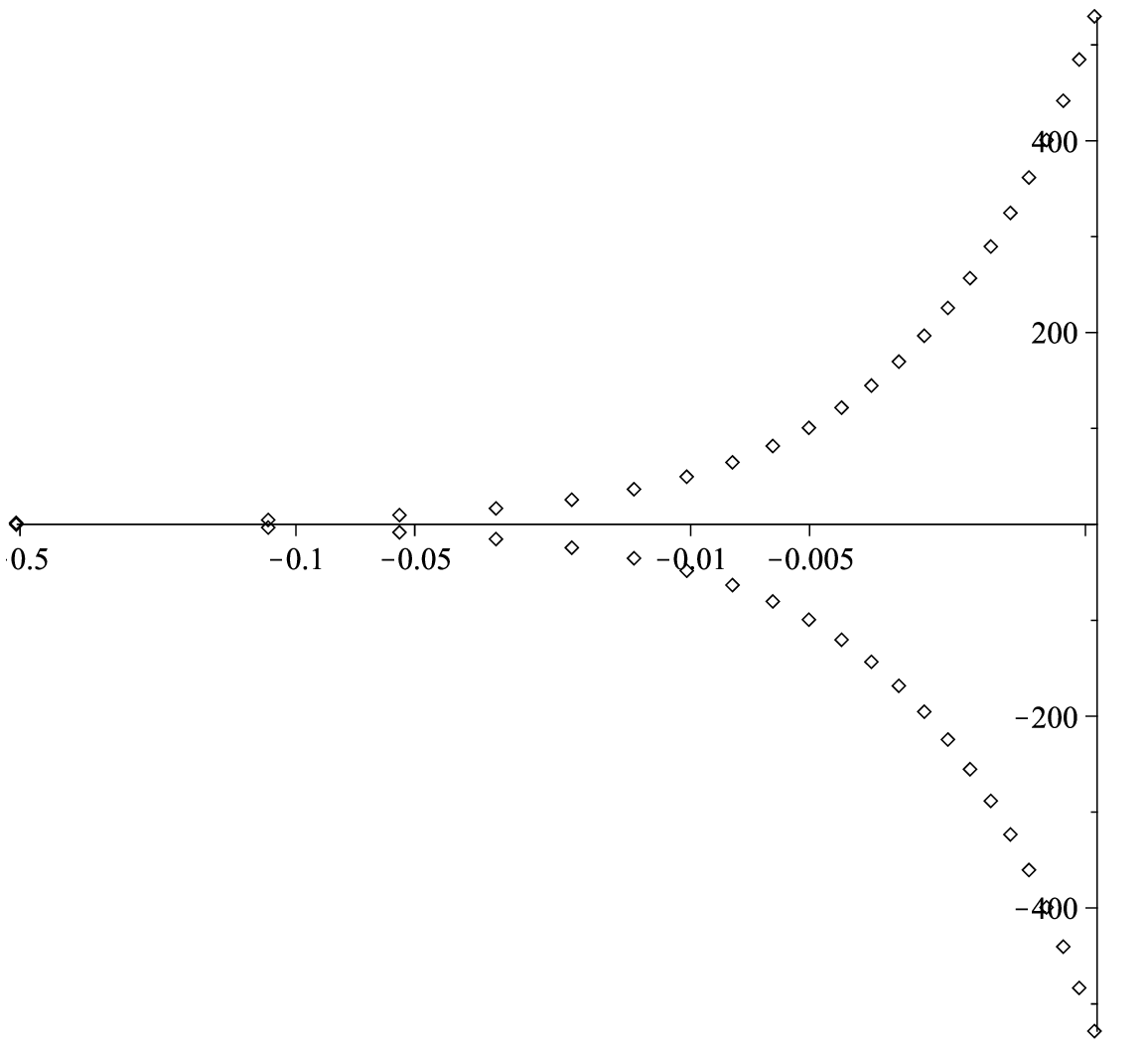}
    \caption{Schematic representation of the eigenvalues distribution}
	\label{fig:Spectrum}
\end{figure}

For visualization of the spectrum of the operator $A$ we have restricted the system to $j=\overline{1,23}$, and calculated eigenvalues taking $\omega_j=j^2$, $c_j=\frac1j$, see Fig.~\ref{fig:Spectrum}.


\section*{Appendix}\label{sec:apx}

\begin{proof}[Proof of Proposition~\ref{prp:nlfl}.]
    The resolvent is a well-defined bounded operator for all $\lambda\in{\mathbb C}$ except at the eigenvalues, so operator $A$ does not have a residual or continuous spectrum.
    It can be easily shown that the characteristic equation does not have purely imaginary roots.
    Since $f(\lambda)$ is analytic in its domain,  it is clear that the set of eigenvalues of $A$ does not have accumulation points.

    The function $f(\lambda)$ possesses singularities at points $\lambda=0$, $\lambda=\pm i\omega_j$, $j\in\mathbb N$.
    Let us investigate its local behaviour around some point $\lambda=i\omega_k$.
    From~\eqref{eq:chareq},
    $$\dfrac2{\gamma |\lambda|} = \left| \sumf{j}\dfrac{c_j^2}{\omega_j} \left(\frac1{\lambda-i\omega_j} - \frac1{\lambda+i\omega_j}\right) \right|
        \leq \sumf{j}\dfrac{c_j^2}{\omega_j}\,\cdot\, \max\left\{ \sup\limits_{j\in{\mathbb N}} \dfrac1{|\lambda-i\omega_j|},\; \sup\limits_{j\in{\mathbb N}} \dfrac1{|\lambda+i\omega_j|} \right\},$$
    so either $\inf\limits_{j\in{\mathbb N}} |\lambda-i\omega_j| \leq \breve R_\lambda$, or
    $\inf\limits_{j\in{\mathbb N}} |\lambda+i\omega_j| \leq \breve R_\lambda$, depending on where $\lambda$ is located at the upper half or bottom half complex plane.
    Thus, for every $\lambda$ satisfying~\eqref{eq:chareq} there exists $k\in{\mathbb N}$, such that $|\lambda\pm i\omega_k| \leq \breve R_\lambda$.
\end{proof}

\begin{proof}[Proof of Proposition~\ref{prp:spr}.]
    We verify that
    \begin{equation*}
        \dfrac{c_k^2}{\gamma \omega_k^2} \cdot \dfrac1{|F'(i\omega_k)|^2} < \dfrac{|F'(i\omega_k)|}{4M}
    \end{equation*}
    provided by $4M < \dfrac{\omega_k |F'(i\omega_k)|^2}{c_k^2}$.
    Indeed,
    \begin{equation*}
        4M < \dfrac{\gamma \omega_k^2 |F'(i\omega_k)|^3}{c_k^2}
    \end{equation*}
    since the upper bound of $4M$ is less than the right-hand side of the previous inequality.
    In order to check this, we calculate
    $$|F'(i\omega_k)|^2 = \left(2\sumk\dfrac{c_j^2}{\omega_j^2-\omega_k^2} + \dfrac{c_k^2}{2\omega_k^2}\right)^{\!\!2} + \dfrac4{\gamma^2\omega_k^2},$$
    that is obviously followed by $\gamma \omega_k |F'(i\omega_k)| > 1$.
\end{proof}

\begin{proof}[Proof of Lemma~\ref{le:R}.]
    Let us show that statement of Lemma~\ref{le:Rouche} holds for $M$ satisfying~\eqref{eq:Mneq2} with $\omega_k > 1$.
    (I.e. \eqref{eq:Mneq2} implies~\eqref{eq:Mneq1}, assuming $\omega_k > 1$.)
    That is,
    $$\dfrac{c_k^2}{\omega_k |F'(i\omega_k)|^2} < \dfrac{c_k^2}{4\gamma\omega_k|F'(i\omega_k)|^3}\,\bigl(\gamma\omega_k|F'(i\omega_k)|+1\bigr)^2,$$
    or
    $$\dfrac{\bigl(\gamma\omega_k|F'(i\omega_k)|+1\bigr)^2}{4\gamma|F'(i\omega_k)|} > 1.$$
    This quadratic inequality with respect to $|F'(i\omega_k)|$ is satisfied as long as $\omega_k > 1$.

    According to Lemma~\ref{le:Rouche}, $\sqrt{R_k}$ is greater than the left end of interval ${\cal I}$, so we consider the following inequality:
    \begin{equation*}
        \left(b-\sqrt{b^2-c}\right)^2 < \frac12|{\rm Re}\lambda_k^*|,
    \end{equation*}
    which is equivalent to
    \begin{equation}\label{eq:rdn2}
        2\,\sqrt{\dfrac1{4M}\left(\dfrac{|F'(i\omega_k)|^2}{4M} - \dfrac{c_k^2}{\omega_k}\right)} > \dfrac{|F'(i\omega_k)|}{2M} - \dfrac{c_k^2}{\omega_k |F'(i\omega_k)|} \left(1+\dfrac1{\gamma \omega_k\,|F'(i\omega_k)|}\right).
    \end{equation}

    Here we have to verify that
    $$\dfrac{|F'(i\omega_k)|}{2M} > \dfrac{c_k^2}{\omega_k |F'(i\omega_k)|} \left(1+\dfrac1{\gamma \omega_k\,|F'(i\omega_k)|}\right),$$
    which is equivalent to
    $$M < \dfrac{\gamma \omega_k^2 |F'(i\omega_k)|^3}{2c_k^2}\,\bigl(\gamma\omega_k|F'(i\omega_k)|+1\bigr)^2.$$
    The last inequality is fulfilled due to~\eqref{eq:Mneq1} and the estimate $\gamma \omega_k |F'(i\omega_k)| > 1$ obtained in the Proof of Proposition~\ref{prp:spr}.

    Thus, right-hand side of~\eqref{eq:rdn2} is strictly positive provided by~\eqref{eq:Mneq1}, and~\eqref{eq:rdn2} implies
    \begin{equation*}
        \dfrac1M > \dfrac{c_k^2}{\gamma \omega_k |F'(i\omega_k)|^3}\,\bigl(\gamma\omega_k|F'(i\omega_k)|+1\bigr)^2.
    \end{equation*}
    Thus, assuming~\eqref{eq:Mneq2}, we have $R_k \in \left( \left(b-\sqrt{b^2-c}\,\right)^2; \; \frac12|{\rm Re}\lambda_k^*|\right]$.

\end{proof}

\begin{remark}
    The condition $\omega_k > 1$ is not particularly restrictive. In fact, to obtain the main result, we examine the asymptotic behavior of the resolvent along the imaginary axis for large $k$. As assumed, $\{\omega_j\}_{j=1}^\infty$ form an unbounded increasing sequence, so there a number $N$ such that $\omega_n > 1$ for all $n \geq N$.
\end{remark}

\begin{proof}[Proof of Proposition~\ref{prp:O}]
    From
    $$\lambda_k^* = -\frac{c_k^2}{\omega_k} \left[2i\sumk \frac{c_j^2}{\omega_j^2-\omega_k^2} + \frac{ic_k^2}{2\omega_k^2} + \frac2{\gamma\omega_k}\right]^{-1} + i\omega_k,$$
    we see that
    $${\rm Re}\lambda_k^* = \dfrac{2\,\gamma\, c_k^2}{4+\gamma^2\,\omega_k^2\left(2\sumk\dfrac{c_j^2}{\omega_j^2-\omega_k^2} + \dfrac{c_k^2}{2\omega_k^2}\right)^{\!\!2}}
    \underset{k\to\infty}= {\cal O} \left(c_k^2 \right)$$
    and
    $${\rm Im}\lambda_k^* = \omega_k + \dfrac{c_k^2}{\omega_k} \cdot \dfrac{2\sumk\dfrac{c_j^2}{\omega_j^2-\omega_k^2} + \dfrac{c_k^2}{2\omega_k^2}}{\dfrac4{\gamma^2\omega_k^2} + \left(2\sumk\dfrac{c_j^2}{\omega_j^2-\omega_k^2} + \dfrac{c_k^2}{2\omega_k^2} \right)^{\!\!2}}
    \underset{k\to\infty}= \omega_k + {\cal O} \left(c_k^2 \right).$$
    Note that $|\lambda_k| \leq |\lambda_k^*| + R_k \leq |\lambda_k^*| + \frac12|{\rm Re} \lambda_k^*|$, which yields the statement.
\end{proof}

\section*{Acknowledgement}
This work was partially supported by a grant from the Simons Foundation (Award 1160640, Presidential Discretionary--Ukraine Support Grants, J.~Kalosha).

\bibliographystyle{aomplain}
\bibliography{Bib}
\end{document}

%% file: Macro.tex
\numberwithin{equation}{section}
\newtheorem{definition}{Definition}[section]
\newtheorem{theorem}{Theorem}[section]
\newtheorem{lemma}{Lemma}[section]
\newtheorem{statement}{Statement}[section]
\newtheorem{proposition}{Proposition}[section]
\newtheorem{remark}{Remark}[section]
\newtheorem{corollary}{Corollary}[section]
\newtheorem{assumption}{Assumption}[section]

\newcommand{\sumf}[1]{\sum\limits_{#1=1}^\infty}
\newcommand{\sumk}{\sum\limits_{j\neq k}}
\newcommand{\qdn}{\omega_j^2+\lambda^2}

%% file: Preprint.bbl
\providecommand{\bysame}{\leavevmode\hbox to3em{\hrulefill}\thinspace}
\providecommand{\noopsort}[1]{}
\providecommand{\mr}[1]{\href{http://www.ams.org/mathscinet-getitem?mr=#1}{MR~#1}}
\providecommand{\zbl}[1]{\href{http://www.zentralblatt-math.org/zmath/en/search/?q=an:#1}{Zbl~#1}}
\providecommand{\jfm}[1]{\href{http://www.emis.de/cgi-bin/JFM-item?#1}{JFM~#1}}
\providecommand{\arxiv}[1]{\href{http://www.arxiv.org/abs/#1}{arXiv~#1}}
\providecommand{\doi}[1]{\url{https://doi.org/#1}}
\providecommand{\MR}{\relax\ifhmode\unskip\space\fi MR }
\providecommand{\MRhref}[2]{%
  \href{http://www.ams.org/mathscinet-getitem?mr=#1}{#2}
}
\providecommand{\href}[2]{#2}
\begin{thebibliography}{10}

\bibitem{BEPS2006}
\bgroup\scshape{}A.~B{\'a}tkai\egroup{}, \bgroup\scshape{}K.~J. Engel\egroup{},
  \bgroup\scshape{}J.~Pr{\"u}ss\egroup{}, and
  \bgroup\scshape{}R.~Schnaubelt\egroup{}, Polynomial stability of operator
  semigroupss,  \emph{Mathematische Nachrichten} \textbf{279 (13-14)} (2006),
  1425--1440. \doi{10.1002/mana.200410429}.

\bibitem{BPS2019}
\bgroup\scshape{}C.~Batty\egroup{}, \bgroup\scshape{}L.~Paunonen\egroup{}, and
  \bgroup\scshape{}D.~Seifert\egroup{}, Optimal energy decay for the wave-heat
  system on a rectangular domain,  \emph{SIAM Journal on Mathematical Analysis}
  \textbf{51 (2)} (2019), 808--819. \doi{10.1137/18M1195796}.

\bibitem{BCT2016}
\bgroup\scshape{}C.~J. Batty\egroup{}, \bgroup\scshape{}R.~Chill\egroup{}, and
  \bgroup\scshape{}Y.~Tomilov\egroup{}, Fine scales of decay of operator
  semigroups,  \emph{J. Eur. Math. Soc.} \textbf{18 (4)} (2016), 853--929.
  \doi{10.4171/JEMS/605}.

\bibitem{BD2008}
\bgroup\scshape{}C.~Batty\egroup{} and \bgroup\scshape{}T.~Duyckaerts\egroup{},
  Non-uniform stability for bounded semi-groups on \textsc{B}anach spaces,
  \emph{Journal of Evolution Equations} \textbf{8} (2008), 765--780.
  \doi{10.1007/s00028-008-0424-1}.

\bibitem{BCPQ2024}
\bgroup\scshape{}G.~Bautista\egroup{}, \bgroup\scshape{}V.~Cabanillas\egroup{},
  \bgroup\scshape{}L.~Potenciano-Machado\egroup{}, and
  \bgroup\scshape{}T.~{Quispe M{\'e}ndez}\egroup{}, Decay rates of strongly
  damped infinite laminated beams,  \emph{Journal of Mathematical Analysis and
  Applications} \textbf{536 (2)} (2024). \doi{10.1016/j.jmaa.2024.128229}.

\bibitem{BT2010}
\bgroup\scshape{}A.~Borichev\egroup{} and \bgroup\scshape{}Y.~Tomilov\egroup{},
  Optimal polynomial decay of functions and operator semigroups,  \emph{Math.
  Ann.} \textbf{347} (2010), 455--478. \doi{10.1007/s00208-009-0439-0}.

\bibitem{BH2007}
\bgroup\scshape{}N.~Burq\egroup{} and \bgroup\scshape{}M.~Hitrik\egroup{},
  Energy decay for damped wave equations on partially rectangular domains,
  \emph{Mathematical research letters} \textbf{14 (1)} (2007), 35--47.

\bibitem{CQMS2022}
\bgroup\scshape{}V.~Cabanillas~Zannini\egroup{},
  \bgroup\scshape{}T.~Quispe~M{\'e}ndez\egroup{}, and
  \bgroup\scshape{}J.~S{\'a}nchez~Vargas\egroup{}, Optimal polynomial stability
  for laminated beams with \textsc{K}elvin--\textsc{V}oigt damping,
  \emph{Mathematical Methods in the Applied Sciences} \textbf{45 (16)} (2022),
  9578--9601. \doi{10.1002/mma.8324}.

\bibitem{CST2020}
\bgroup\scshape{}R.~Chill\egroup{}, \bgroup\scshape{}D.~Seifert\egroup{}, and
  \bgroup\scshape{}Y.~Tomilov\egroup{}, Semi-uniform stability of operator
  semigroups and energy decay of damped waves,  \emph{Philosophical
  Transactions of the Royal Society A} \textbf{378 (2185)} (2020).
  \doi{10.1098/rsta.2019.0614}.

\bibitem{DRV2024}
\bgroup\scshape{}C.~Deng\egroup{}, \bgroup\scshape{}J.~Rozendaal\egroup{}, and
  \bgroup\scshape{}M.~Veraar\egroup{}, Improved polynomial decay for unbounded
  semigroups,  \emph{Preprint} (2024). \doi{10.48550/arXiv.2407.09323}.

\bibitem{D2007}
\bgroup\scshape{}T.~Duyckaerts\egroup{}, Optimal decay rates of the energy of a
  hyperbolic--parabolic system coupled by an interface,  \emph{Asymptotic
  Analysis} \textbf{51 (1)} (2007), 17--45.

\bibitem{Fu2008}
\bgroup\scshape{}X.~Fu\egroup{}, Logarithmic decay of hyperbolic equations with
  arbitrary boundary damping,  \emph{Preprint} (2008).
  \doi{10.48550/arXiv.0805.0625}.

\bibitem{GW2019}
\bgroup\scshape{}B.-Z. Guo\egroup{} and \bgroup\scshape{}J.-M. Wang\egroup{},
  \emph{Control of Wave and Beam PDEs. The Riesz Basis Approach},
  \emph{Communications and Control Engineering}, Springer Cham, 2019.
  \doi{10.1007/978-3-030-12481-6}.

\bibitem{KZB2021}
\bgroup\scshape{}J.~Kalosha\egroup{}, \bgroup\scshape{}A.~Zuyev\egroup{}, and
  \bgroup\scshape{}P.~Benner\egroup{}, On the eigenvalue distribution for a
  beam with attached masses,  in \emph{Stabilization of Distributed Parameter
  Systems: Design Methods and Applications}
  (\bgroup\scshape{}G.~Sklyar\egroup{} and \bgroup\scshape{}A.~Zuyev\egroup{},
  eds.), \emph{SEMA SIMAI Springer Series} \textbf{2}, Springer International
  Publishing, Cham, 2021, pp.~43--56. \doi{10.1007/978-3-030-61742-4\_3}.

\bibitem{Lebeau1996}
\bgroup\scshape{}G.~Lebeau\egroup{}, {\'E}quation des ondes amorties,  in
  \emph{Algebraic and Geometric Methods in Mathematical Physics. Mathematical
  Physics Studies} (\bgroup\scshape{}A.~B. de~Monvel\egroup{} and
  \bgroup\scshape{}V.~Marchenko\egroup{}, eds.), \textbf{19}, Springer,
  Dordrecht, 1996, pp.~73--109. \doi{10.1007/978-94-017-0693-3\_4}.

\bibitem{Liu2005}
\bgroup\scshape{}Z.~Liu\egroup{} and \bgroup\scshape{}B.~Rao\egroup{},
  Characterization of polynomial decay rate for the solution of linear
  evolution equation,  \emph{Zeitschrift f{\"u}r angewandte Mathematik und
  Physik ZAMP} \textbf{56} (2005), 630--644. \doi{10.1007/s00033-004-3073-4}.

\bibitem{LR2007}
\bgroup\scshape{}Z.~Liu\egroup{} and \bgroup\scshape{}B.~Rao\egroup{},
  Frequency domain approach for the polynomial stability of a system of
  partially damped wave equations,  \emph{Journal of mathematical analysis and
  applications} \textbf{335 (2)} (2007), 860--881.
  \doi{10.1016/j.jmaa.2007.02.021}.

\bibitem{LGM1999}
\bgroup\scshape{}Z.-H. Luo\egroup{}, \bgroup\scshape{}B.-Z. Guo\egroup{}, and
  \bgroup\scshape{}{\"O}.~Morg{\"u}l\egroup{}, \emph{Stability and
  Stabilization of Infinite Dimensional Systems with Applications},
  Springer-Verlag, London, 1999.

\bibitem{Phung2007}
\bgroup\scshape{}K.~D. Phung\egroup{}, Polynomial decay rate for the
  dissipative wave equation,  \emph{Journal of Differential Equations}
  \textbf{240 (1)} (2007), 92--124. \doi{10.1016/j.jde.2007.05.016}.

\bibitem{RZZ2005}
\bgroup\scshape{}J.~Rauch\egroup{}, \bgroup\scshape{}X.~Zhang\egroup{}, and
  \bgroup\scshape{}E.~Zuazua\egroup{}, Polynomial decay for a
  hyperbolic--parabolic coupled system,  \emph{Journal de Math{\'e}matiques
  Pures et Appliqu{\'e}es} \textbf{84 (4)} (2005), 407--470.
  \doi{10.1016/j.matpur.2004.09.006}.

\bibitem{R2023}
\bgroup\scshape{}J.~Rozendaal\egroup{}, Operator-valued $(\sc{L}^p, \sc{L}^q)$
  \textsc{F}ourier multipliers and stability theory for evolution equations,
  \emph{Indagationes Mathematicae} \textbf{34 (1)} (2023), 1--36.
  \doi{10.1016/j.indag.2022.08.008}.

\bibitem{RSS2019}
\bgroup\scshape{}J.~Rozendaal\egroup{}, \bgroup\scshape{}D.~Seifert\egroup{},
  and \bgroup\scshape{}R.~Stahn\egroup{}, Optimal rates of decay for operator
  semigroups on \textsc{H}ilbert spaces,  \emph{Advances in Mathematics}
  \textbf{346} (2019), 359--388. \doi{10.1016/j.aim.2019.02.007}.

\bibitem{Tomilov2001}
\bgroup\scshape{}Y.~Tomilov\egroup{}, A resolvent approach to stability of
  operator semigroups,  \emph{Journal of Operator Theory} \textbf{46 (1)}
  (2001), 63--98.

\bibitem{ZK2023ECC}
\bgroup\scshape{}A.~Zuyev\egroup{} and \bgroup\scshape{}J.~Kalosha\egroup{}, A
  dynamic observer for a class of infinite--dimensional vibrating flexible
  structures,  in \emph{2023 European Control Conference (ECC)}, IEEE,
  Bucharest, 2023, pp.~200--205. \doi{10.23919/ECC57647.2023.10178223}.

\bibitem{zuev2006partial}
\bgroup\scshape{}A.~Zuyev\egroup{}, Partial asymptotic stability of abstract
  differential equations,  \emph{Ukrainian Mathematical Journal} \textbf{58}
  no.~5 (2006), 709--717. \doi{10.1007/s11253-006-0096-3}.

\bibitem{zuyev2003partial}
\bgroup\scshape{}A.~Zuyev\egroup{}, Partial asymptotic stability and
  stabilization of nonlinear abstract differential equations,  in \emph{42nd
  IEEE International Conference on Decision and Control}, IEEE, 2003,
  pp.~1321--1326. \doi{10.1109/CDC.2003.1272792}.

\bibitem{zuyev2005stabilization}
\bgroup\scshape{}A.~Zuyev\egroup{} and \bgroup\scshape{}O.~Sawodny\egroup{},
  Stabilization of a flexible manipulator model with passive joints,
  \emph{IFAC Proceedings Volumes} \textbf{38} no.~1 (2005), 784--789.
  \doi{10.3182/20050703-6-CZ-1902.00531}.

\end{thebibliography}
